\documentclass[10pt]{article}

\usepackage{amsmath} \usepackage{amsfonts}
\usepackage{mathrsfs}  \usepackage{amsthm}
\usepackage[normalem]{ulem}   \usepackage{enumerate}
\usepackage{graphicx}  \usepackage{epstopdf}
\usepackage{subfigure}  \usepackage{lineno}
\usepackage{multirow,bm}
\usepackage{color}
\usepackage{indentfirst}
\usepackage{algorithm}
\usepackage{algorithmic}

\usepackage{tabularx}
\newtheorem{proposition} {Proposition}[section]

\newtheorem{theorem}{Theorem}[section]
\newtheorem{lemma}{Lemma}[section]

\theoremstyle{definition}
\newtheorem{definition}{Definition}[section]
\newtheorem{remark}{Remark}[section]

 \newcommand{\nn}{\nonumber}
\def \[{\begin{equation}} \def \]{\end{equation}}

\def \diag{\mbox{diag}}

\begin{document}

\title{An  Inexact Manifold Augmented Lagrangian Method for Adaptive Sparse Canonical Correlation Analysis with Trace Lasso Regularization\thanks{This research was supported by  the Natural Science Foundation of China (NSFC 11571074, 61672005).}}
\date{ }

\author{Kang-Kang Deng $^\ddag$, ~~Zheng Peng \footnote{School of Mathematics and Computational Science, Xiangtan University, Xiangtan 411105, China. Corresponding author, E-mail: pzheng@xtu.edu.cn. }  \footnote{College of Mathematics and Computer Science, Fuzhou University,  Fuzhou  350108,   China. }
 }

\maketitle

\noindent {\textbf{Abstract:}   Canonical correlation analysis (CCA for short) describes the relationship between two sets of variables by finding some linear combinations of these variables that maximizing the correlation coefficient. However, in high-dimensional settings where the number of variables exceeds   sample size, or in the case of that the variables are highly correlated, the traditional CCA is no longer appropriate. In this paper,  an adaptive sparse version of CCA (ASCCA for short) is proposed by  using the trace Lasso regularization. The proposed ASCCA reduces the  instability of the estimator when the covariates are highly correlated,  and thus improves its interpretation.  The ASCCA   is further reformulated to an optimization problem on Riemannian manifolds, and an manifold inexact augmented Lagrangian method is then proposed for   the resulting  optimization problem.  The performance of the  ASCCA  is compared with the other sparse CCA techniques in different simulation settings, which illustrates that the ASCCA is feasible and efficient.



\numberwithin{equation}{section}

\section{Introduction}

Canonical correlation analysis (CCA for short)   firstly proposed by Hotelling \cite{Hotelling1992Relations}  aims to tackle the associations between two sets of variables. It   has wide applications in many important  fields such as  biology \cite{Witten2009Extensions, Parkhomenko2008Sparse, Lin2014Correspondence}, medicine \cite{Correa2008Canonical},  image analysis \cite{Fu2008Image, Loog2004Dimensionality}, etc.

Suppose that there are two data sets:  $X\in\mathbb{R}^{n\times p}$ containing $p$ variables and $Y\in\mathbb{R}^{n\times q}$ containing $q$ variables, both are obtained from $n$ observations.  The CCA seeks two linear combinations of these variables from   $X$ and $Y$ with maximal  correlation coefficient. Specially, let $$\Sigma_{xx} = \frac{1}{n}X^TX, \Sigma_{yy} = \frac{1}{n}Y^TY$$ be the sample covariance matrices of $X$ and $Y$ respectively, and $\Sigma_{xy} = \frac{1}{n}X^TY$ be the sample cross-covariance matrix,  then the CCA finds a pair $(u, v)$  such that
\begin{equation}\label{eq:def}
\begin{split}
\mbox{corr}(Xu,Yv)=\frac{u^T\Sigma_{xy}v}{\sqrt{u^T\Sigma_{xx}u}\sqrt{v^T\Sigma_{yy}v}}
\end{split}
\end{equation}
is maximized.  The new variables $u$ and $v$ are  canonical variables,  and the correlations between   canonical variables are called canonical correlations. The canonical variables  $u$  and $v$ can be respectively obtained by the eigenvectors of  matrix
$$\Sigma_{xx}^{-1}\Sigma_{xy}\Sigma_{yy}^{-1}\Sigma_{yx}.$$
The canonical correlations are given by the positive square root of those eigenvalues.
Since the  CCA model \eqref{eq:def} is in form of fractional, it is difficult to optimize. A equivalent formulation of CCA is given by
\[\label{originalCCA}
\left\{\begin{split}
\max_{u,v}& ~ u^T\Sigma_{xy} v \\
\mbox{s.t. } & u^T\Sigma_{xx} u = 1, v^T\Sigma_{yy} v = 1,
\end{split}\right.
\]
which can be regarded  as an optimization problem on the generalized Stiefel manifolds.  However, a potential disadvantage of the CCA is that, the learned solution  is a linear combination of all  original variables,  which brings down the interpretability. If  the  number of variables   exceeds   sample size,  traditional CCA cannot be performed due to that $\Sigma_{xx}$ and $\Sigma_{yy}$ are singular.  Hence, many researchers proposed various sparse CCA (SCCA) to handle the case that the number of variables  exceeds   observations, and to improve  the interpretability of   canonical variables by restricting the linear combinations to a subset of  original variables.

In this paper, we  propose an adaptive sparse CCA model by incorporating the trace Lasso regularization. The  matrix version of trace Lasso regularization  can be adopted to   both highly correlated and uncorrelated data.  Our major contributions are summarized in follows:
\begin{enumerate}
  \item[1)] We present a matrix version of trace Lasso regularization, and show that the new regularization function enjoys the properties of original trace Lasso.
  \item[2)] By introducing trace Lasso regularization into the CCA model, we obtain an adaptive sparse CCA model (ASCCA). To our knowledge, our ASCCA is the first to takes the data correlation into account in the CCA model. In addition, our model consider multiple variables simultaneously.
  \item[3)] The new model is reformulated to  an optimization problem on the generalized Stiefel manifold. An manifold inexact augmented Lagrangian method is proposed for the resulted optimization problem, and the convergence is established under some assumptions. 
  \item[4)] The experimental results demonstrate that, the proposed ASCCA is superior to some existing sparse CCA models.
\end{enumerate}

The rest of the paper is organized as follows. Section \ref{related} briefly  gives some reviews on the related works.  Section \ref{sec-trace} proposes an adaptive sparse version of the CCA introduced by  the new trace Lasso regularization. Section \ref{sec-opt} provides an optimization reformulation and the manifold inexact augmented Lagrangian  method for the new model, and gives the convergence analysis. In Section \ref{sec-simu},  a simulation study is provided to show the validity  and efficiency of the proposed method.  Section \ref{sec-con} concludes this paper with some final remarks.
\section{Related works}\label{related}

It is well known that, if the sample size exceeds dimension, the traditional CCA does not perform. To overcome this difficulty, various methods were proposed via incorporate different regularization function.
Vinod \cite{Vinod2006Canonical} proposed a canonical ridge, which is an adaptation of  the ridge regression for the CCA framework proposed by Hoerl and Kennard \cite{Hoerl2000Ridge}, and introduced an efficient sparsity penalty strategy.  After that, various approaches for  sparse CCA (SCCA for short)  were proposed in  literature, which includes $\ell_1$ regularization\cite{Parkhomenko2009Sparse,Witten2009A}, elastic net \cite{Waaijenborg2007Quantifying}, group sparse and structured sparse \cite{Lin2013Group,Chen2012Structured},  etc. There also exists some limitations.  If there is a group of variables which the pairwise correlation is high, the Lasso tends to select only one variable from this group, which may lead  some misunderstands to the truth. Group sparse regularization needs the prior knowledge of group, which is unrealistic in some real applications. The proposed   adaptive sparse CCA model utilized the new trace Lasso regularization, which incorporates data matrix into regularization, to adaptively deal with the correlation of covariation matrix.


  The original SCCA model is difficult to handle, so many researchers simplify it by assuming that $\Sigma_x$ and $\Sigma_y$ are diagonal matrices or identity matrices. Parkhomenko, et al  \cite{Parkhomenko2009Sparse} assume that,  the covariance matrices $\Sigma_{xx},\Sigma_{yy} $ are the identity matrices, and used a sparse singular value decomposition to derive sparse singular vectors. Wilms and Croux \cite{Wilms2015Sparse} converted the SCCA model into a penalized regression framework. Suo \cite{suo2017sparse}  presented an approximated  SCCA model as follows
  \[\label{relax1}\left\{
  \begin{split}
  \min_{u,v} & -u^T\Sigma_{xy}v + \lambda_1\|u\|_1 + \lambda_2 \|v\|_1 \\
  \mbox{s.t. } ~ & u^T\Sigma_{xx}u \leq 1, v^T\Sigma_{xx}v \leq 1,
  \end{split}\right.
  \]
  and   problem \eqref{relax1} was solved  by a  linearized alternating direction method of multipliers (LADMM).  Witten \cite{Witten2009A} further relaxed \eqref{relax1}  to
 \[\label{relax2}\left\{
 \begin{split}
 \min_{u,v} & -u^T\Sigma_{xy}v \\
 \mbox{s.t. } & \|u\|_2^2 \leq 1,\|v\|_2^2 \leq 1,P_1(u)\leq c_1, P_2(v) \leq c_2,
 \end{split}\right.
 \]
 where $P_1(u)$ and $P_2(v)$ are some regularizations for sparsity,  and then developed a penalized matrix decomposition algorithm to solve model \eqref{relax2}.
Focusing on a sparse version of the original CCA model \eqref{originalCCA},
Gao \cite{gao2017sparse}  proposed a two-stage method by a convex relaxation of  CCA model.  For the matrix case, many researchers  adopted the residual model to obtain the high-order canonical variables \cite{Parkhomenko2009Sparse,Witten2009A,suo2017sparse}.  In this paper,  we  get  multiple variables simultaneously in our new model. In addition, all results on the matrix case mentioned above  have not given convergence analysis for their algorithms, we proposed an efficient method to solve our new model, and provided the convergence analysis.

The original trace Lasso was proposed by Grave \cite{Grave2011Trace}.  Trace Lasso regularization was successful applied to various scenarios  including subspace clustering \cite{wang2015robust}, sparse representation classification \cite{wang2014robust} and subspace segment \cite{lu2013correlation}, and so on. However, they only considered the  trace Lasso regularization in vector case in literature. In this paper, we generalize the original trace Lasso regularization to   matrix case,  and adopt it as a new regularization for  the SCCA,   and get an adaptive SCCA model.
\paragraph{Notations:}
We use capital and lowercase symbols to represent matrix and vector, respectively. Let $1_d\in\mathbb{R}^d$ denote  the vector of all 1's, $e_i$ be a vector whose $i$-th entry is 1 and 0 for others, $\mbox{Diag}(w)$ is a diagonal matrix where the $i$-th diagonal entry is $w_i$, and $\mbox{diag}(W)$ be a vector where the $i$-th entry is $W_{ii}$. Let $W_{i\cdot}$ and $W_{\cdot j}$ denote the $i$-th row and $j$-th column of $W$,  $tr(W)$ be the trace of $W$. For a vector $w$, let $\|w\|_1,\|w\|_2$ be the $\ell_1$ and $\ell_2$ norm. For a matrix $W\in\mathbb{R}^{n\times p}$, let  $\|W\|_{2,1} = \sum_{i=1}^{n}\|W_{i\cdot}\|_2$ be the  $\ell_{2,1}$ norm, $\|W\|_F$ and $\|W\|_*$ denote the Frobenius norm and nuclear norm respectively, $\|W\|_{op}$ denote the operator norm.

\section{Adaptive sparse CCA using trace Lasso regularization}\label{sec-trace}
\subsection{Trace Lasso in vector case}

Consider the following linear estimator:
\[
\min_w \frac{1}{2}\|Xw - y\| + \lambda \Omega(w)
\]
 where $X\in\mathbb{R}^{n\times p}$ is a data matrix. The trace Lasso  is  a correlation based penalized norm  proposed  by Grave et al \cite{Grave2011Trace}  for  balancing the $\ell_1$  and $\ell_2$ norm.  It is defined as follows
$$\Omega(w)=\|X\diag(w)\|_*$$
where $\| \cdot \|_{*}$ is nuclear norm. A main advantage of trace Lasso being superior  to other norm  is that, the trace Lasso involves the data matrix $X$,   makes it adaptive to the correlation of data. As shown in \cite{Grave2011Trace}, if  each column of $X$ is normalized to $1$, the trace Lasso interpolates between the $\ell_1$ norm and $\ell_2$ norm in the sense of
\[\label{relaxofTr} \|w\|_{1}\leq\|X\diag(w)\|_{*}\leq\|w\|_{2}.\]
The inequality are tight. To see this, if the data are uncorrelated ($X^TX = I_p$), trace Lasso reduce to $\|w\|_1$, and if the data are  highly correlated ($X = X_{\cdot 1}1^T$),  trace Lasso equals to $\|w\|_2$. 


\subsection{Trace Lasso in matrix case}

Let   $W\in\mathbb{R}^{p\times r}$,   define a linear operator $ \mathcal{A}_{X}: \mathbb{R}^{p\times r} \rightarrow\mathbb{R}^{n\times pr}$ as
$$
\mathcal{A}_{X}(W) = \left(
X\mbox{Diag}(W_{\cdot 1}),\cdots,X\mbox{Diag}(W_{\cdot r})
\right)
$$
and its adjoint operator $\mathcal{A}_{X}^*:\mathbb{R}^{n\times pr} \rightarrow \mathbb{R}^{p\times r}$ as
$$\mathcal{A}_{X}^*(M) = \left(
\mbox{diag}(X^TM_1),\cdots,\mbox{diag}(X^TM_r)
\right)$$
where $M_i = M(:,(i-1)p+1:ip)$ denotes $i$-th block matrix of $M$. Then, the trace Lasso in matrix case is defined as follows
\[\label{TLM}
\Omega(W) = \|\mathcal{A}_{X}(W)\|_*
\]
It is easy to show that,  the   trace Lasso regularizer in matrix case \eqref{TLM} has similar properties to that in vector case. If   each column of $X$ is normalized, then the  linear operator $\mathcal{A}_{X}$ can be rewritten to
\[\label{decomTr}
\mathcal{A}_{X}(W) = \sum_{i=1}^{r}\sum_{j=1}^{p}X_{\cdot j} W_{ij}\bar{e}_{ij}^T =\sum_{j=1}^{p}X_{\cdot j}\cdot \left(\sum_{i=1}^{r} w_{ij}\bar{e}_{ij}^T \right)
\]
where $\bar{e}_{ij}\in\mathbb{R}^{pr\times 1}$ is an unit vector in which the $((i-1)p+j)$-th component is 1 and the others are 0. There are two special case:
\begin{enumerate}
\item[1)]   If the data (i.e., column vectors of $X$) are uncorrelated, i.e.,  $X^TX = I_p$. Then  \eqref{decomTr} gives a singular value decomposition of  $\mathcal{A}_X(W)$. In the case, trace Lasso \eqref{TLM} reduces to  the $\ell_{2,1}$ norm
     $$ \|\mathcal{A}_{X}(W)\|_{*}=\sum_{j=1}^{p}\|X_{\cdot j}\|_{2}\|W_{j \cdot}\|_2=\|W\|_{2,1}.$$
\item[2)]  If the data are highly correlated,  especially if all columns of $X$ are identical and have unit size, we have
$$\mathcal{A}_{X}(W) = X_{\cdot 1}\cdot \sum_{j=1}^{p}\sum_{i=1}^{r} W_{ij}\bar{e}_{ij}^T  = X_{\cdot 1}\cdot (\mbox{vec}(W))^T,$$
where $\mbox{vec}(W) = [W_{\cdot 1};\cdots;W_{\cdot r}]$. Then    trace Lasso \eqref{TLM} reduces to Frobenius norm
    $$ \|\mathcal{A}_{X}(W)\|_{*}=\|X_{\cdot 1}\cdot (\mbox{vec}(W))^T\|_{*}=\|X_{\cdot 1}\|_{2}\cdot \|\mbox{vec}(W)\|_2 = \|W\|_F.$$
\end{enumerate}
The following proposition show that the trace Lasso \eqref{TLM} in matrix case is adaptive to the correlation of data, which is similar to the original trace Lasso. 
\begin{proposition}\label{proposition ineq}
Let $X\in\mathbb{R}^{n\times p}$, and each column of $X$ is normalized,    $W\in\mathbb{R}^{p\times r}$. Then 
  \[\label{eqofTL}
  \|W\|_{F} \leq \|\mathcal{A}_{X}(W)\|_* \leq \sqrt{r} \|W\|_{2,1}.
  \]
\end{proposition}
\begin{proof}
We first show that  $\|\mathcal{A}_{X}(W)\|_F  = \| W\|_F  $. Specifically, we have
$$
\|\mathcal{A}_{X}(W)\|_F^2 = \sum_{i=1}^{r}\|X\mbox{diag}(W_{\cdot i})\|_F^2 = \sum_{i=1}^{r}\sum_{j=1}^{p}W_{ij}^2\|X_{\cdot j}\|_2^2 = \sum_{i=1}^{r}\sum_{j=1}^{p}W_{ij}^2 = \|W\|_F^2
$$
Then, for the first inequality of \eqref{eqofTL}  we have
 $$
  \|W\|_F  = \|\mathcal{A}_{X}(W)\|_F \leq  \|\mathcal{A}_{X}(W)\|_*.
 $$
 Denote the $j$-th column of the $i$-th submatrix in $M$ by $M^i_{:,j}$, and let  $\hat{M}_j = [M^1_{:,j},M^2_{:,j},\cdots, M^r_{:,j}] \in\mathbb{R}^{n\times r}$, then  for the second inequality  of \eqref{eqofTL} we have
  \[
  \begin{split}
      \|\mathcal{A}_{X}(W)\|_*  & = \max_{\|M\|_{op}\leq 1} \left<M, \mathcal{A}_{X}(W) \right> \\
                                & = \max_{\|M\|_{2}\leq 1} \left< \mathcal{A}_{X}^*(M),W \right> \\
                                & = \max_{\|M\|_{2}\leq 1} \sum_{i=1}^{r} W_{\cdot i}^T\mbox{diag}(X^TM^i) \\
                                & = \max_{\|M\|_{2}\leq 1}\sum_{j=1}^{p}X_{\cdot j}^T \left(\sum_{i=1}^{r} W_{ij} M^i_{\cdot j} \right)  \\
                                & \leq \max_{\|M\|_{2}\leq 1}\sum_{j=1}^{p} \|X_{\cdot j}\|_2 \cdot \left\|\sum_{i=1}^{r} W_{ij} M_{\cdot j}^i \right\|_2\\
                                & \leq  \max_{\|M\|_{2}\leq 1}\sum_{j=1}^{p} \|X_{\cdot j}\|_2 \cdot \|\hat{M}_j\|_F\|W_{j \cdot}\|_2 \\
                                &  \leq \sqrt{r} \sum_{j=1}^{p} \|W_{j \cdot }\|_2 =  \sqrt{r}\|W\|_{2,1}.
  \end{split}\nonumber
  \]
  The first equality used the fact that the dual norm of the trace norm is the operator norm. The last inequality used  that   $\|X_{.j}\|=1 (\forall j)$,  and $\|M^i_{:,j}\|_2 \leq 1$ which deduces  $\|\hat{M}_j\|_F\leq \sqrt{r}$.
 \end{proof}
\begin{remark}
 If $r=1$, then Proposition \ref{proposition ineq} is indeed  Proposition 3 in \cite{Grave2011Trace}.
\end{remark}

\subsection{Regression framework of  the adaptive SCCA}
Given two data matrices $X\in \mathbb{R}^{n\times p}$ and $Y \in \mathbb{R}^{n\times q}$  on the same set of observations, where $n$ is the sample size, $p$ and $q$ are the feature numbers. Without loss of generality, we assume that data matrices $X$ and $Y$ are mean centered. By $\Sigma_{xx} = \frac{1}{n}X^TX, \Sigma_{YY} = \frac{1}{n}Y^TY$ and $\Sigma_{xy} = \frac{1}{n}X^TY$, the CCA problem can be rewritten as
\begin{equation}\label{eq:def2}\left\{
\begin{split}
&(u^*, v^*)=\arg\max_{u, v}~u^{T}X^{T}Yv, \\
&\mbox{s.t.} \quad u^{T}X^{T}Xu=1, v^{T}Y^{T}Yv=1.
\end{split}\right.
\end{equation}
For multiple canonical vectors, let $U=[u_1,u_2,\cdots, u_r]$ and $V=[v_1,v_2,\cdots, v_r]$  where $(u_i, v_i)$ denote the $i$-th pair of the canonical vectors,
  the multiple CCA problem  is
 \begin{equation}\label{eq:def3} \left\{
\begin{split}
&(U^*, V^*)=\arg\max_{U, V} ~ tr(U^{T}X^{T}YV), \\
&\mbox{s.t.} \quad U^{T}X^{T}XU=I_r, V^{T}Y^{T}YV=I_r
\end{split}\right.
\end{equation}

The CCA problem \eqref{eq:def3} can be reformulated  to a constrained  bilinear regression problem of the form
 \[\label{eq:def6}\left\{
\begin{split}
&(U^*, V^*)=\arg\min_{U,V} \frac{1}{2}\|XU-YV\|_F^2 \\
&\mbox{s.t.}  \quad U^{T}X^{T}XU=I_r, V^{T}Y^{T}YV=I_r.
\end{split}\right. \]

 To adapt to the dependence of data, we consider an adaptive sparse CCA (SCCA) model with trace Lasso regularization. Specifically, we have
\[\label{atlscca}\left\{
\begin{split}
&(U^*, V^*)=\arg\min_{U,V} \frac{1}{2}\|XU-YV\|_F^2+ \lambda_u \|\mathcal{A}_X(U)\|_*  +\lambda_v \|\mathcal{A}_Y(V)\|_*, \\
&\mbox{s.t.}  \quad U^T(X^TX)U=I_r, ~V^T(Y^TY)V=I_r,
\end{split}\right. \]
where $U\in \mathbb{R}^{p\times r}, V\in \mathbb{R}^{q\times r}$, and $\lambda_u, \lambda_v$ are the penalty parameters,  $\mathcal{A}_X:\mathbb{R}^{p\times r} \rightarrow \mathbb{R}^{n \times pr}$ and $\mathcal{A}_Y:\mathbb{R}^{q\times r} \rightarrow \mathbb{R}^{n \times qr}$ are linear operators.

\section{Optimization method for SCCA \eqref{atlscca}}\label{sec-opt}
The SCCA model \eqref{atlscca} is a nonconvex and nonsmooth optimization problem,  and it is difficult to be solved. Riemannian optimization methods are popular to solve a class constrained optimization  problem with special structure.  Hence, in this section we first reformulate problem \eqref{atlscca} to a nonsmooth optimization  problem on the generalized Stiefel manifolds, then adopt an manifold inexact augmented Lagrangian method in \cite{dkk2019ALM} to solve the resulting problem. Finally, we give a convergence analysis of the proposed method.
\subsection{Augmented Lagrangian scheme}\label{ALM}
Let $\mathcal{M}_1 = \left\{U| U^TX^TXU = I_r\right\}, \mathcal{M}_2 = \left\{V| V^TY^TYV = I_r\right\}$, and $g(\cdot) = \lambda_u\|\cdot\|_*, h(\cdot) = \lambda_v\|\cdot\|_*$, then problem \eqref{atlscca} can be reformulated as
\[\label{Mani_scca}\left\{
\begin{split}
&(U^*, V^*)=\arg\min_{U,V} \frac{1}{2}\|XU-YV\|_F^2+ g(\mathcal{A}_X(U))  +h(\mathcal{A}_Y(V)), \\
&\mbox{s.t.}  \quad U\in\mathcal{M}_1, ~V\in\mathcal{M}_2,
\end{split}\right. \]
Here, we assume that $X^TX$ and $Y^TY$ are positive define \footnote{If it is not positive define, we can replace $X^TX$ by $(1-\alpha)X^TX + \alpha I_p$.}.
Then $\mathcal{M}_1$ and $\mathcal{M}_2$ can be regarded to  generalized Stiefel manifolds, and problem \eqref{Mani_scca} is an optimization problem on generalized Stiefel manifolds. We further reformulate \eqref{Mani_scca} to
\[\label{constran_scca}\left\{
\begin{split}
(U^*, V^*) & =\arg\min_{U,V} \frac{1}{2}\|XU-YV\|_F^2+ g(P)  +h(Q), \\
\mbox{s.t.} & \quad \mathcal{A}_X(U) = P, \mathcal{A}_Y(V) = Q \\
            &    \quad U\in\mathcal{M}_1, ~V\in\mathcal{M}_2,
\end{split}\right. \]
The Lagrangian function associated with \eqref{constran_scca} is given by
\[\label{Lagrangian}
\begin{split}
\mathcal{L}(U,V,P,Q;\Lambda_1,\Lambda_2) = & \frac{1}{2}\|XU-YV\|_F^2+ g(P)  +h(Q) \\
                                          & - \left<\Lambda_1,\mathcal{A}_X(U) -P\right> - \left<\Lambda_2,\mathcal{A}_Y(V) - Q\right>,
                                          \end{split}
\]
where $\Lambda_1$ and $\Lambda_2$ denote the Lagrangian multipliers. Let $\rho$ be a penalty parameter. Then, the corresponding augmented Lagrangian function is given by
\[\label{Aug_Lagrangian}
 \mathcal{L}_{\rho}(U,V,P,Q;\Lambda_1,\Lambda_2) =   \mathcal{L}(U,V,P,Q;\Lambda_1,\Lambda_2)
                            +\frac{\rho}{2} \|\mathcal{A}_X(U) -P\|_F^2 + \frac{\rho}{2}\|\mathcal{A}_Y(V) - Q\|_F^2
 \]

Then, the proposed manifold inexact augmented Lagrangian method for \eqref{constran_scca} is summarized in Algorithm \ref{alg2}.
\begin{algorithm}[htb]
   \caption{Manifold inexact augmented Lagrangian method for problem \eqref{constran_scca}}
   \label{alg2}
\begin{algorithmic}[1]
   \STATE {\bfseries Input: } Let $\Lambda_{\min}< \Lambda_{\max}$,  $X_0\in\mathcal{M}$, tolerance $\epsilon_{\min}\ge 0$, $\epsilon_0>0$,  $\rho_0>1$, $\mu>1, 0<\tau<1$. 
   \FOR{$k=0,1,\cdots$}
   \STATE Updating the primary variables by approximately solving 
   \begin{equation}\label{approx}
     (U^{k+1},V^{k+1},P^{k+1},Q^{k+1}) = \arg\min_{\substack{U\in\mathcal{M}_1,V\in\mathcal{M}_2,\\P,Q}} \left\{ \Psi_k: =  \mathcal{L}_{\rho_k}(U,V,P,Q;\Lambda_1^k,\Lambda_2^k)\right\}
   \end{equation}
  such that a specified stopping criteria is hit.  
   \STATE Updating the dual variables via 
   \[\nn
   \begin{array}{l}
     \Lambda_1^{k+1}     = \Lambda_1^k - \rho_k(\mathcal{A}_X(U^{k+1}) -  P^{k+1}),~~\Lambda_2^{k+1}     = \Lambda_2^k - \rho_k(\mathcal{A}_Y(V^{k+1}) -  Q^{k+1}) \\

   \end{array}
   \]
   \STATE Updating $\rho_{k+1}$ via $$\rho_{k+1} = \left\{
   \begin{array}{cc}
     \rho_k & \mbox{if }~\|R^k_i\|_{\infty} \leq \tau \|R^{k-1}_i\|_{\infty}, i=1,2, \\
     \gamma \rho_k & \mbox{otherwise},
   \end{array}
    \right.  $$
  where $R^k_1 = \mathcal{A}_X(U^{k+1}) -  P^{k+1}, R^k_2 = \mathcal{A}_Y(V^{k+1}) -  Q^{k+1}$.

   \ENDFOR
\end{algorithmic}
\end{algorithm}

\subsection{Convergence analysis}
Let $W = (U, V, P, Q)\in\mathbb{R}^{p\times r} \times \mathbb{R}^{q\times r} \times \mathbb{R}^{n\times pr} \times \mathbb{R}^{n\times q r}$ be  a variable formed by concatenating $U,V,P$ and $Q$. Let $F(W) = f(U, V)+g(P)+g(Q)$ where $f(U, V)= \frac{1}{2}\|XU-YV\|_F^2$.  Then, problem \eqref{constran_scca} can be rewritten as
\[\label{newpro}
\min_W F(W) ~~ \mbox{s.t. }~~ h(W) = 0, ~W\in\mathcal{N},
\]
 where $\mathcal{N}:=\mathcal{M}_1\times \mathcal{M}_2 \times \mathbb{R}^{n\times pr} \times \mathbb{R}^{n\times q r}$,  and $h(W) $ is given by
 \[\nn
 h(W): = \left(
   \begin{array}{c}
     \mathcal{A}_X(U) - P~,  ~\mathcal{A}_Y(V) - Q
   \end{array}
 \right).
 \]
 The corresponding augmented Lagrangian can be rewritten as
\[
\mathcal{L}_{\rho}(W;Z) = F(W) + \sum_{i,j}^{r}Z_{ij}[h(W)]_{ij} + \frac{\rho}{2}\sum_{i,j}^{r} [h(W)]_{ij}^2
\]
The corresponding KKT condition is given by
\begin{equation}\label{kkt}
  0\in \partial F(W^*) +  \sum_{i=1}^{m}\sum_{j=1}^{r}Z_{ij} \mbox{grad}[h(W^*)]_{ij}, ~~ h(W^*) = 0, ~~W^*\in\mathcal{N}
\end{equation}
 where $ \partial F(W^*) $  is the Riemannian subdifferential of $F$ at $W^*$. 
To obtain an efficient implementation of Algorithm 1, we  inexactly solve the iteration subproblem \eqref{approx} in which    the following stoping criteria is used:
\[\label{stoping 1}
 \delta^k \in \partial  \Psi_k(W^{k+1}) ~~ \mbox{and} ~\|\delta^k\| \leq \epsilon_k \]
where $\epsilon_k \rightarrow 0$ as $k\rightarrow \infty$.

Following Yang, Zhang and Song \cite{Yang2014}, we give the constraint qualifications of problem \eqref{newpro}:

\begin{definition}[LICQ]
  Linear independence constraint qualifications (LICQ) are said to hold at $W\in\mathcal{N}$ for problem \eqref{newpro} if
  \[
  \{\mbox{grad}[h(W)]_{ij}| i=1,\cdots,m;j=1,\cdots,r\} \mbox{ are linearly independent in } T_W\mathcal{N}.
  \]
\end{definition}

\begin{theorem}
  Suppose $\{W^k\}_{k\in\mathbb{N}}$ is a sequence generated by Algorithm \ref{alg2}, and the stoping criteria \eqref{stoping 1}
 is hit at the $k$-th iteration. Then the limit point set of $\{W^k\}_{k\in\mathbb{N}}$ is nonempty. Let $W^*$ be a limit point of  $\{W^k\}$, and the LICQ holds at $W^*$.  Then $W^*$ is a KKT point of the problem \eqref{newpro}.
\end{theorem}
\begin{proof}
See \cite{dkk2019ALM}.
\end{proof}
%

\begin{lemma}
 The LICQ always holds at $\forall ~W\in \mathcal{N}$ for problem \eqref{newpro} .
\end{lemma}

\begin{proof}
Let $h_1(W): = \mathcal{A}_X(U)-P, h_2(W): = \mathcal{A}_Y(V) - Q$, then
\[
\begin{split}
\nabla [h_1(W)]_{ij} & = \nabla_U [h_1(W)]_{ij} \times \nabla_V [h_1(W)]_{ij} \times \nabla_P [h_1(W)]_{ij} \times \nabla_Q [h_1(W)]_{ij},\\
&i=1,\cdots,n; j=1,\cdots,pr, \\
\nabla [h_2(W)]_{ij} & = \nabla_U [h_2(W)]_{ij} \times \nabla_V [h_2(W)]_{ij} \times \nabla_P [h_2(W)]_{ij} \times \nabla_Q [h_2(W)]_{ij},\\
&i=1,\cdots,n; j=1,\cdots,qr.
\end{split}
\]
For all   $i=1,\cdots,m, j=1,\cdots,n$, let $E^{m\times n}_{ij}$ be a $m\times n$  matrix in which the entry at the $i$-th row and the $j$-column is 1,  the  others are 0.  Then
\[
\begin{array}{lll}
 \nabla_P [h_1(W)]_{ij} = E^{n\times pr}_{ij},  & \nabla_Q [h_1(W)]_{ij} = 0, &i=1,\cdots,n; j=1,\cdots,pr, \\
 \nabla_P [h_1(W)]_{ij} = 0,                    & \nabla_Q [h_1(W)]_{ij} = E^{n\times qr}_{ij}, &i=1,\cdots,n; j=1,\cdots,qr.
\end{array}
\]
A basis of the normal cone of $\mathcal{M}:\mathcal{M}_1\times \mathcal{M}_2$ at $(U,V)$, denoted by $N_{U}\mathcal{M}_1 \times N_{V}\mathcal{M}_2$,  is given by
\begin{equation}
  \left\{\Sigma_{xx} U(e_ie_j^T + e_je_i^T): i=1,\cdots,r,j=i,\cdots,r \right\} \times \left\{\Sigma_{yy} V(e_ie_j^T + e_je_i^T): i=1,\cdots,r,j=i,\cdots,r \right\} .\nonumber
\end{equation}
It is easy to  show that, $\forall~W\in\mathcal{N}$, if there exists $Z^1, Z^2$ such that
\[\label{licq}
\sum_{i=1}^{n}\sum_{j=1}^{pr} Z_{ij}^1\nabla [h_1(W)]_{ij} + \sum_{i=1}^{n}\sum_{j=1}^{qr} Z_{ij}^2\nabla [h_2(W)]_{ij} \in N_W\mathcal{N},
\]
then    $Z^1 = Z^2 = 0$.  Since  $\mathcal{N}$ is a submanifold of  Euclidean space, it derives immediately from  \eqref{licq} that  \[\nn
\sum_{i=1}^{n}\sum_{j=1}^{pr} Z_{ij}^1\mbox{grad} [h_1(W)]_{ij} + \sum_{i=1}^{n}\sum_{j=1}^{qr} Z_{ij}^2\mbox{grad} [h_2(W)]_{ij}=0.
\]
Which implies  that   LICQ  holds at  $W$ and  completes the proof.
\end{proof}

\subsection{Riemannian gradient method for subproblem \eqref{approx}}
In section \ref{ALM}, we present an manifold inexact augmented Lagrangian method to solve problem \eqref{constran_scca}.
The main challenge in the proposed method (Algorithm \ref{alg2}) is  to solve   subproblem \eqref{approx} efficiently.  Problem \eqref{approx} is a nonsmooth problem under manifold constrained. In this subsection, we first get an equivalence smooth optimization problem by using the Moreau envelop technique, then we present Riemannian gradient method to solve the equivalent problem.

The proximal mapping $\mbox{Prox}_p(\cdot)$ associated with $p$ is defined by
\[\label{prox}
\mbox{Prox}_p(U) = \arg\min_W\left\{p(W) + \frac{1}{2}\|U-W\|_F^2\right\}.
\]
For   fixed $\Lambda_1,\Lambda_2$ and $\rho$, we consider
\[\nn
 \min_{U,V,P,Q} \Big\{\Psi(U,V,P,Q):   = \mathcal{L}_{\rho}(U,V,P,Q;\Lambda_1,\Lambda_2),~~
\mbox{s.t. } ~   U\in\mathcal{M}_1, V\in\mathcal{M}_2\Big\}
 \]
Let
\[\nn
\begin{split}
\psi(U,V) & = \inf_{P,Q}\Psi(U,V,P,Q) \\
          & =  \frac{1}{2}\|XU-YV\|_F^2+ g(\mbox{Prox}_{g/\rho}(\mathcal{A}_X(U)- \Lambda_1/\rho))  +h(\mbox{Prox}_{h/\rho}(\mathcal{A}_Y(V)- \Lambda_2/\rho)) \\
          & +\frac{\rho}{2} \|\mathcal{A}_X(U)- \Lambda_1/\rho - \mbox{Prox}_{g/\rho}(\mathcal{A}_X(U)- \Lambda_1/\rho)\|_F^2 \\
          &+ \frac{\rho}{2}\|\mathcal{A}_Y(V) - \mbox{Prox}_{h/\rho}(\mathcal{A}_Y(V)- \Lambda_2/\rho)\|_F^2  - \frac{1}{2\rho}\|\Lambda_1\|_F^2 - \frac{1}{2\rho}\|\Lambda_2\|_F^2.
\end{split}
\]
Hence, if
\[\nn
\begin{split}
(\tilde{U},\tilde{V},\tilde{P},\tilde{Q}) & = \arg\min_{U\in\mathcal{M}_1, V\in\mathcal{M}_2,P,Q} \Psi(U,V,P,Q), \\
\end{split}
\]
then $(\tilde{U},\tilde{V},\tilde{P},\tilde{Q})$ can be computed by
\[\label{separable}
\left\{
\begin{array}{l}
  (\tilde{U},\tilde{V}) = \arg\min\limits_{U\in\mathcal{M}_1, V\in\mathcal{M}_2}\psi(U,V),  \\
  \tilde{P} = \mbox{Prox}_{g/\rho}(\mathcal{A}_X(\tilde{U})- \Lambda_1/\rho), \\
  \tilde{Q} = \mbox{Prox}_{h/\rho}(\mathcal{A}_Y(\tilde{V})- \Lambda_2/\rho).
\end{array}
\right.
\]
Notice that the subproblems for $P$ and $Q$ are proximal operators. Both  $g$ and $h$ in \eqref{separable} are nuclear norm functions, the proximal operator is indeed a singular value shrinkage operator, which is given by:
\[\label{svs}
D_{\tau}(Y)  = \arg\min_X\left\{\frac{1}{2}\|X-Y\|_F^2 + \tau \|X\|_* \right\}
            = UD_{\tau}(\Sigma)V^{'},\]
where $Y = U\Sigma V^{'}, \Sigma = \mbox{diag}(\{\sigma_i\}_{1\leq i\leq r})$ and $D_{\tau}(\Sigma) = \mbox{diag}(\{(\sigma_i-\tau)_{+}\})$.

Now we focus on the subproblem regarding jointed variable  $(U,V)$  in \eqref{separable}. Recall that
\[\label{subproblem1}
  (\tilde{U},\tilde{V})   = \arg\min_{U,V}\big\{\psi(U,V) ,  ~
    \mbox{s.t. }    U\in\mathcal{M}_1, V\in\mathcal{M}_2\big\}.
\]
Let $\mathcal{W}: = (U,V)$ and $\mathcal{M}: = \mathcal{M}_1 \bigotimes \mathcal{M}_2$ be  a product manifold. Then, problem \eqref{subproblem1} can be formulated
\[\label{subproblem2}
   \tilde{\mathcal{W}} = \arg\min_{\mathcal{W}}\big\{\psi(\mathcal{W}), ~      \mbox{s.t. }    \mathcal{W}\in\mathcal{M}\big\}.
 \]
By Lemma \ref{lem-morea},   $\psi(\mathcal{W})$ is continuously differentiable in Euclidean space, and its Euclidean gradient is
\begin{equation}\nn
\nabla \psi(\mathcal{W}) = \left(
  \begin{array}{c}
    \nabla_U\psi(\mathcal{W}) \\
    \nabla_V\psi(\mathcal{W})
  \end{array}
  \right)= \left(
  \begin{array}{ccc}
  X^T(XU - YV) + \rho\mathcal{A}^*_X(U)\left[\mathcal{A}_X(U) - \frac{1}{\rho}\Lambda_1 - \mbox{Prox}_{g/\rho}(\mathcal{A}_X(U) - \frac{1}{\rho}\Lambda_1)\right] \\
  Y^T(YV - XU) + \rho\mathcal{A}^*_Y(V)\left[\mathcal{A}_Y(V) - \frac{1}{\rho}\Lambda_2 - \mbox{Prox}_{h/\rho}(\mathcal{A}_Y(V) - \frac{1}{\rho}\Lambda_2)\right]
  \end{array}
  \right)
\end{equation}
Since $\mathcal{M}$ is a Riemannian submanifold in Euclidean space,   by lemma \eqref{retract-euclidean}    $\psi(W)$   is retraction smooth, and its Riemannian gradient is
\begin{equation}
  \mbox{grad}\psi(\mathcal{W}) = \left(
  \begin{array}{c}
    \mbox{grad}_U\psi(\mathcal{W}) \\
    \mbox{grad}_V\psi(\mathcal{W})
  \end{array}
  \right) =
  \left(
  \begin{array}{c}
    P_{T_U\mathcal{M}_1}(\nabla_U \psi(\mathcal{W})) \\
    P_{T_V\mathcal{M}_2}(\nabla_V \psi(\mathcal{W}))
  \end{array}
  \right).
\end{equation}

It is shown that problem \eqref{subproblem2} is a smooth optimization  problem on Riemannian manifold.  In this paper,  we adopt a Riemannian Barzilai-Borwein (RBB) gradient  method  \cite{IannazzoThe} to  solve problem \eqref{subproblem2}, see Algorithm \ref{alg3} for details.

\begin{algorithm}[tb]
   \caption{Riemannian Barzilai-Borwein gradient method for subproblem \eqref{subproblem2}, RBB }
   \label{alg3}
\begin{algorithmic}[1]
   \STATE {\bfseries Given: }   $\mathcal{W}^0\in\mathcal{M}$,   tolerance $\epsilon>0$, initial step size $\alpha_0^{BB}$.  Let  $g^0 = \mbox{grad}\psi(\mathcal{W}^0)$, the sufficient decrease parameter $\gamma$ and the step length contraction factor $\sigma\in(0,1)$.
   \STATE {\bfseries Initialize:  } k = 0.

   \WHILE{$\|g^k\| \geq \epsilon$}
   \STATE Find the smallest positive integer $h$ such that
   \[
   \psi(R_{\mathcal{W}^k}(-\sigma^h\alpha_k^{BB}g^k)) \leq \psi(\mathcal{W}^k) - \gamma \sigma^h \alpha_k^{BB} \|g^k\|_{\mathcal{W}^k}^2
   \]
   and set $\alpha_k : = \sigma^h\alpha_k^{BB}$.
   \STATE Compute $     \mathcal{W}^{k+1} = R_{\mathcal{W}^k}(-\alpha_k g^k)$, and $g^{k+1} = \mbox{grad}\psi(\mathcal{W}^{k+1})$.
   \STATE Set
   $$
   \tau_{k+1} = \frac{\left<s^k,s^k\right>_{\mathcal{W}^{k+1}}}{\left<s^k,y^k\right>_{\mathcal{W}^{k+1}}},
   $$
   where $s^k: = g^{k+1} - \mathcal{T}_{\mathcal{W}^k\rightarrow \mathcal{W}^{k+1}}(g^k)$ and $y^k:=\mathcal{T}_{\mathcal{W}^k\rightarrow \mathcal{W}^{k+1}}(-\alpha_kg^k)$.
   \STATE Compute the new step size $\alpha_{k+1}^{BB}$
   $$
   \alpha_{k+1}^{BB} = \left\{
   \begin{array}{ll}
     \min\{\alpha_{\max},\max\{\alpha_{\min},\tau_{k+1}\}\} & \mbox{if} \left<s^k,y^k\right>_{\mathcal{W}^{k+1}} >0, \\
     \alpha_{\max} & \mbox{otherwise}.
   \end{array}
   \right.
   $$
   \STATE Set $k:=k+1$.
   \ENDWHILE
\end{algorithmic}
\end{algorithm}

\section{Random Simulation} \label{sec-simu}

In the section, the performance of the SCCA model and the proposed method is verified by random simulation.   The proposed  adaptive trace Lasso regularization CCA  in this paper is compared  with the sparse CCA-$\ell_1$ model (named as CoLaR) proposed by Gao \cite{gao2017sparse}. The CoLaR is a computationally feasible two-stage method, which consists of a convex-programming-based initialization stage and a group-Lasso-based refinement stage. In the first stage, CoLaR solves the following convex minimization problem:
\[\label{init}\left\{
\begin{split}
  \hat{F}  &= \arg\min_F -\left<\hat{\Sigma}_{xy},F\right> + \tau \|F\|_1 \\
    & \mbox{s.t.} ~ \|(\hat{\Sigma}_x)^{1/2}F(\hat{\Sigma}_y)^{1/2}\|_* \leq r, \|(\hat{\Sigma}_x)^{1/2}F(\hat{\Sigma}_y)^{1/2}\|_{2}\leq 1,
\end{split}\right.
\]
where $\hat{\Sigma}_{xy},\hat{\Sigma}_{x}$ and $\hat{\Sigma}_{y}$ are  sample covariance matrices. Let $U^0$ and $V^0$ be the matrices whose column vectors are respectively the top $r-$ left and right singular vectors of $\hat{F}$. Then, in the second stage, CoLaR solve the following group-Lasso problem:
\[\label{lar}\left\{
\begin{split}
U^1 = \arg\min_U tr(U^T\hat{\Sigma}_xU) - 2tr(U^T\hat{\Sigma}_{xy}V^0) + \tau^{'}\sum_{i=1}^{p} \|U_i.\|_2, \\
V^1 = \arg\min_V tr(V^T\hat{\Sigma}_yV) - 2tr((U^0)^T\hat{\Sigma}_{xy}V) + \tau^{'}\sum_{i=1}^{q} \|V_i.\|_2
\end{split}\right.
\]
Finally, $U^1$ and $V^1$ are projected to generalized Stiefel manifolds $U^T\hat{\Sigma}_xU = I_r$ and $V^T\hat{\Sigma}_yV = I_r$.

The codes of CoLaR were downloaded from Ma's homepage\footnote{http://www-stat.wharton.upenn.edu/~zongming/research.html}. The first stage is terminated if it does not make much progress, or it reaches the maximum iteration number 10, the parameter in step \eqref{init} is set to $\tau =0.55\sqrt{(r+\log(q))/n}$.   The obtained solution $(U^0, V^0)$ is used as the initial point for CoLaR.  For our method, we also use $(U^0,V^0)$ obtained by \eqref{init} as the original point. The parameters of our Algorithm \ref{alg2} are set to : $\tau = 0.99, \sigma = 1.05, \rho_0 = \max\{\lambda_{\max}(X^TX),\lambda_{\max}(Y^TY)\}, Z_{\min} = -100\cdot 1_{d\times r}, Z_{\max} = 100\cdot 1_{d\times r}, Z^0 = 0_{d\times r}$ and $\epsilon_k = \max(10^{-3},0.9^k)$  where $k\in\mathbb{N}$ is the iteration counter. We terminated   Algorithm \ref{alg2}  if  $\max\{\|\mathcal{A}_X(U^{k+1}) -  P^{k+1}\|_F^2, \|\mathcal{A}_Y(V^{k+1}) -  Q^{k+1}\|_F^2\} \leq 10^{-8}$ or $k\ge 500$. For Algorithm \ref{alg3}, the inner iteration of the  Algorithm \ref{alg2}, it is terminated if  $\|\delta^k\| \leq \epsilon_k $, where $\delta^k \in \partial  \Psi_k(W^{k+1})$, or the inner iteration number exceeds $100$.

\paragraph{Parameters selection}
We using the $\kappa$-fold cross-validation (CV) to select the optimal parameters for CoLaR and our method. For our method, we set $\lambda_u  = b\sqrt{(r+\log(p))/n}$ and $\lambda_v = b\sqrt{(r+\log(q))/n}$. For CoLaR, the parameter $\tau'$ is set to = $\tau' =  b\sqrt{(r+\log(\max(p,q)))/n}$. Then, we used $\kappa$-fold  cross validation to select a common penalty parameter $b$. In particular, for each choice of parameter $b$,  a $\frac{\kappa-1}{\kappa}$ proportion of the data (training sample) is used to obtain estimates $(\hat{U},\hat{V})$. Then we evaluate the correlation between obtained canonical vectors in the remaining $\frac{1}{\kappa}$ of the data set (testing sample), and compute average correlation over $k$ CV steps. The optimal parameter $b$ then corresponds to the highest average correlation.  We set $\kappa=10$ in this paper.



\subsection{Simulation data }
The random simulation data is generated with respect to the following  scheme: for each simulation design, there has $p$ variables in data set $X$ and $q$ variables in data set $Y$, and the sample size is $n$. In all these settings, we set $r =2$ and  $\Lambda =
\mbox{Diag}(\lambda_1,\lambda_2)$ where $\lambda_1 = 0.9,\lambda_2 = 0.8$. The nonzero rows of both U and V are set  to $S = \{1, 6, 11, 16, 21\}$, the corresponding value at the nonzero coordinates are generated by normalizing random numbers drawn from the uniform distribution on finite set $\{-2, -1, 0, 1, 2\}$. We determinate the covariance matrix $\Sigma_x$ and $\Sigma_y$ via the following three procedure:
\begin{enumerate}
  \item[(1)] Identity matrices: $\Sigma_x = I_p, \Sigma_y = I_q$,
  \item[(2)] Toeplitz matrices: $\Sigma_x = 0.3^{|i-j|}$,  for $i,j\in[p]$, and $\Sigma_y = 0.3^{|i-j|}$,  for $i,j\in[q]$,
  \item[(3)] Correlation matrices: $$
  (\Sigma_x)_{ij} = \left\{
  \begin{array}{ccc}
    \sigma & (i,j)\in S ~ \mbox{and} ~ i\neq j \\
    1 & i == j \\
    0 & \mbox{otherwise}
  \end{array}
  \right. ~~(\Sigma_y)_{ij} = \left\{
  \begin{array}{ccc}
    \sigma & (i,j)\in S ~ \mbox{and} ~ i\neq j \\
    1 & i == j \\
    0 & \mbox{otherwise}
  \end{array}
  \right.,
  $$
  where $\sigma$ is the correlation degree. In our experiments, we set $\sigma = 0.3, 0.5, 0.8$, respectively.
\end{enumerate}
After determinate the covariance matrix, we generate the cross-covariance matrix $\Sigma_{xy}$ by $$\Sigma_{xy} = \Sigma_x U \Lambda V \Sigma_y, $$
and  generate the data sets $X\in\mathbb{R}^{n\times p}$ and $Y\in\mathbb{R}^{n\times q}$ via
$$
\left(\begin{array}{c}
        x \\
        y
      \end{array}
\right) \sim \mathcal{N}\left(
\left(\begin{array}{c}
        0 \\
        0
      \end{array}
\right),\left(
          \begin{array}{cc}
            \Sigma_x & \Sigma_{xy} \\
            \Sigma_{xy}^T & \Sigma_y \\
          \end{array}
        \right)
\right)
$$
Finally, we used the subspace distance between the estimation $(\widehat{U},\widehat{V})$ and the ground truth $(U,V)$ as the prediction
errors:
$$
\mbox{lossu} = \|P_U - P_{\widehat{U}}\|_F^2,~~ \mbox{lossv} = \|P_V - P_{\widehat{V}}\|_F^2.
$$
where $P_U = U(U^TU)^{-1}U^T,  P_V = V(V^TV)^{-1}V^T $ are projection matrices on the column space of $U$ and $V$, respectively.

Tables \ref{table:1},\ref{table:2},\ref{table:3}display the simulation results:  in each settings, the medians of the prediction error of CoLaR and our method over 100 repetitions for four different configurations of $(n; p; q)$ values are list for comparison.

In each table, the columns ``lossu'' and ``lossv'' report the smallest estimation errors of the medians out of the eleven trials on each simulated dataset,  $(\rho_1,\rho_2)$ is the two canonical correlations.  The columns ``Init'' report the results generated by the initialization step \eqref{init}.  The columns ``CoLaR'' and ``Ours'' report the  results of the CoLaR   and our method,  respectively.

In case (1) and (2), i.e.,  $\Sigma_x$ and $\Sigma_y$ are the identity matrices and Toeplitz matrices, the results of  CoLaR and our method are roughly the same. In case (3), our method   consistently outperform  the CoLaR estimators. In particular, as the correlation parameter $\sigma$ going to larger, the performance of our method is   more significant  superiority than the CoLaR. This is in accordance to the theoretical results for trace Lasso regularization in Section \ref{sec-trace}.
\begin{table}[tbp]
\label{table:1}
\caption{ Performance for our method and CoLaR  in case (1) }
\vskip 10pt

\centering
\setlength{\tabcolsep}{0.4mm}{
\begin{tabular}{c|c|c|c|c|c|c|c|c|c}
\hline
   \multirow{2}{*}{$(n,p,q)$}           & \multicolumn{3}{c|}{lossu}              & \multicolumn{3}{c|}{lossv}              & \multicolumn{3}{c}{$(\rho_1,\rho_2)$}                                  \\ \cline{2-10}
     & Init  & CoLaR       & Ours            & Init  & CoLaR       & Ours            & Init            & CoLaR       & Ours                      \\ \hline
(200,200,200) & 0.533 & 0.114          & \textbf{0.092} & 0.325 & 0.094          & \textbf{0.063} & (0.8325,0.6662) & (0.8852,0.7726) & \textbf{(0.8891,0.7852)} \\
(300,200,200) & 0.050 & 0.011          & \textbf{0.008} & 0.174 & 0.047          & \textbf{0.045} & (0.8832,0.6571) & (0.8987,0.7974) & \textbf{(0.8988,0.7977)} \\
(400,200,200) & 0.195 & 0.019          & \textbf{0.018} & 0.237 & \textbf{0.030} & 0.031          & (0.8527,0.7436) & (0.8948,0.7937) & \textbf{(0.8949,0.7941)} \\
(300,400,400) & 0.361 & \textbf{0.047} & 0.049          & 0.290 & 0.038          & \textbf{0.037} & (0.8442,0.6944) & (0.8962,0.7859) & (0.8972,0.7857)          \\ \hline
\end{tabular}}
\end{table}

\begin{table}[tbp]
\label{table:2}
\caption{ Performance for our method and  CoLaR  in case (2)}
\vskip 10pt

\centering
\setlength{\tabcolsep}{0.4mm}{
\begin{tabular}{c|c|c|c|c|c|c|ccc}
\hline
\multirow{2}{*}{$(n,p,q)$} & \multicolumn{3}{c|}{lossu}              & \multicolumn{3}{c|}{lossv}              & \multicolumn{3}{c}{$(\rho_1,\rho_2)$}                                                                                      \\ \cline{2-10}
                           & Init  & CoLaR       & Ours            & Init  & CoLaR       & Ours            & \multicolumn{1}{c|}{Init}            & \multicolumn{1}{c|}{CoLaR}                 & Ours                      \\ \hline
(200,200,200)              & 0.249 & \textbf{0.071} & 0.079          & 0.246 & 0.096          & \textbf{0.081} & \multicolumn{1}{c|}{(0.8823,0.6879)} & \multicolumn{1}{c|}{(0.8954,0.7767)}          & \textbf{(0.8965,0.7773)} \\
(300,200,200)              & 0.275 & 0.032          & 0.032          & 0.187 & \textbf{0.022} & 0.027          & \multicolumn{1}{c|}{(0.8711,0.7475)} & \multicolumn{1}{c|}{(0.8925,0.8101)}          & (0.8941,0.8100)          \\
(400,200,200)              & 0.211 & 0.027          & \textbf{0.022} & 0.237 & \textbf{0.022} & 0.030          & \multicolumn{1}{c|}{(0.8517,0.7470)} & \multicolumn{1}{c|}{(0.8962,0.7807)}          & (0.8966,0.7791)          \\
(300,400,400)              & 0.250 & \textbf{0.040} & 0.043          & 0.394 & 0.039          & \textbf{0.038} & \multicolumn{1}{c|}{(0.8525,0.6726)} & \multicolumn{1}{c|}{\textbf{(0.8953,0.7804)}} & (0.8952,0.7796)          \\ \hline
\end{tabular}}
\end{table}

\begin{table}[tbp]
\label{table:3}
\caption{ Performance for our method and  CoLaR in case (3) }
\vskip 10pt

\centering
\setlength{\tabcolsep}{0.4mm}{
\begin{tabular}{ccccccccccc}
\hline                                                                                                                                                                                                                                                                                                                                                                          \hline
\multicolumn{1}{c|}{\multirow{2}{*}{$\sigma$}}                     & \multirow{2}{*}{$(n,p,q)$}              & \multicolumn{3}{|c|}{lossu}                                                                             & \multicolumn{3}{c|}{lossv}                                                                             & \multicolumn{3}{c}{$(\rho_1,\rho_2)$}                                                                                     \\ \cline{3-11}
           \multicolumn{1}{c|}{}        &      & \multicolumn{1}{|c|}{Init}  & \multicolumn{1}{c|}{CoLaR}       & \multicolumn{1}{c|}{Ours}            & \multicolumn{1}{c|}{Init}  & \multicolumn{1}{c|}{CoLaR}       & \multicolumn{1}{c|}{Ours}            & \multicolumn{1}{c|}{Init}            & \multicolumn{1}{c|}{CoLaR}                 & Ours                      \\ \hline
\multicolumn{1}{c|}{\multirow{3}{*}{0.3}} & \multicolumn{1}{c|}{(200,200,200)} & \multicolumn{1}{c|}{0.411} & \multicolumn{1}{c|}{0.111}          & \multicolumn{1}{c|}{\textbf{0.085}} & \multicolumn{1}{c|}{0.760} & \multicolumn{1}{c|}{0.153}          & \multicolumn{1}{c|}{\textbf{0.097}} & \multicolumn{1}{c|}{(0.8555,0.6786)} & \multicolumn{1}{c|}{(0.8846,0.7845)}          & 0.8880,0.7950            \\
\multicolumn{1}{c|}{}                     & \multicolumn{1}{c|}{(300,200,200)} & \multicolumn{1}{c|}{0.384} & \multicolumn{1}{c|}{0.082}          & \multicolumn{1}{c|}{\textbf{0.059}} & \multicolumn{1}{c|}{0.366} & \multicolumn{1}{c|}{0.069}          & \multicolumn{1}{c|}{\textbf{0.042}} & \multicolumn{1}{c|}{(0.8394,0.5961)} & \multicolumn{1}{c|}{(0.8878,0.7686)}          & \textbf{(0.8941,0.7732)} \\
\multicolumn{1}{c|}{}                     & \multicolumn{1}{c|}{(400,200,200)} & \multicolumn{1}{c|}{0.284} & \multicolumn{1}{c|}{0.040}          & \multicolumn{1}{c|}{\textbf{0.035}} & \multicolumn{1}{c|}{0.460} & \multicolumn{1}{c|}{0.056}          & \multicolumn{1}{c|}{\textbf{0.040}} & \multicolumn{1}{c|}{(0.8034,0.7434)} & \multicolumn{1}{c|}{(0.8890,0.7907)}          & (0.8918,0.7870)          \\
\multicolumn{1}{c|}{}                     & \multicolumn{1}{c|}{(300,400,400)} & \multicolumn{1}{c|}{0.416} & \multicolumn{1}{c|}{\textbf{0.050}} & \multicolumn{1}{c|}{0.055}          & \multicolumn{1}{c|}{0.856} & \multicolumn{1}{c|}{0.142}          & \multicolumn{1}{c|}{\textbf{0.104}} & \multicolumn{1}{c|}{(0.7650,0.6883)} & \multicolumn{1}{c|}{(0.8667,0.7874)}          & (0.8737,0.7812)          \\ \hline
\multicolumn{1}{c|}{\multirow{3}{*}{0.5}} & \multicolumn{1}{c|}{(200,200,200)} & \multicolumn{1}{c|}{0.681} & \multicolumn{1}{c|}{0.473}          & \multicolumn{1}{c|}{\textbf{0.187}} & \multicolumn{1}{c|}{1.608} & \multicolumn{1}{c|}{0.390}          & \multicolumn{1}{c|}{\textbf{0.278}} & \multicolumn{1}{c|}{(0.8250,0.5239)} & \multicolumn{1}{c|}{(0.8880,0.6763)}          & (0.8849,0.7199)          \\
\multicolumn{1}{c|}{}                     & \multicolumn{1}{c|}{(300,200,200)} & \multicolumn{1}{c|}{0.365} & \multicolumn{1}{c|}{0.074}          & \multicolumn{1}{c|}{\textbf{0.046}} & \multicolumn{1}{c|}{0.505} & \multicolumn{1}{c|}{0.280}          & \multicolumn{1}{c|}{\textbf{0.161}} & \multicolumn{1}{c|}{(0.8584,0.6357)} & \multicolumn{1}{c|}{(0.8820,0.7840)}          & \textbf{(0.8930,0.7887)} \\
\multicolumn{1}{c|}{}                     & \multicolumn{1}{c|}{(400,200,200)} & \multicolumn{1}{c|}{0.621} & \multicolumn{1}{c|}{0.081}          & \multicolumn{1}{c|}{\textbf{0.054}} & \multicolumn{1}{c|}{0.627} & \multicolumn{1}{c|}{0.101}          & \multicolumn{1}{c|}{\textbf{0.056}} & \multicolumn{1}{c|}{(0.8196,0.7204)} & \multicolumn{1}{c|}{(0.8790,0.7720)}          & \textbf{(0.8852,0.7810)} \\
\multicolumn{1}{c|}{}                     & \multicolumn{1}{c|}{(300,400,400)} & \multicolumn{1}{c|}{1.032} & \multicolumn{1}{c|}{0.732}          & \multicolumn{1}{c|}{\textbf{0.238}} & \multicolumn{1}{c|}{1.082} & \multicolumn{1}{c|}{0.287}          & \multicolumn{1}{c|}{\textbf{0.134}} & \multicolumn{1}{c|}{(0.8480,0.5122)} & \multicolumn{1}{c|}{(0.8512,0.7277)}          & \textbf{(0.8762,0.7597)} \\ \hline
\multicolumn{1}{c|}{\multirow{3}{*}{0.8}} & \multicolumn{1}{c|}{(200,200,200)} & \multicolumn{1}{c|}{2.141} & \multicolumn{1}{c|}{1.524}          & \multicolumn{1}{c|}{\textbf{0.169}} & \multicolumn{1}{c|}{2.226} & \multicolumn{1}{c|}{1.738}          & \multicolumn{1}{c|}{\textbf{0.407}} & \multicolumn{1}{c|}{(0.7891,0.2296)} & \multicolumn{1}{c|}{(0.7850,0.2490)}          & \textbf{(0.8519,0.7941)} \\
\multicolumn{1}{c|}{}                     & \multicolumn{1}{c|}{(300,200,200)} & \multicolumn{1}{c|}{1.511} & \multicolumn{1}{c|}{1.668}          & \multicolumn{1}{c|}{\textbf{0.328}} & \multicolumn{1}{c|}{2.140} & \multicolumn{1}{c|}{0.738}          & \multicolumn{1}{c|}{\textbf{0.059}} & \multicolumn{1}{c|}{(0.7275,0.4971)} & \multicolumn{1}{c|}{(0.7841,0.6782)}          & \textbf{(0.8877,0.7679)} \\
\multicolumn{1}{c|}{}                     & \multicolumn{1}{c|}{(400,200,200)} & \multicolumn{1}{c|}{1.186} & \multicolumn{1}{c|}{0.356}          & \multicolumn{1}{c|}{\textbf{0.085}} & \multicolumn{1}{c|}{1.207} & \multicolumn{1}{c|}{0.497}          & \multicolumn{1}{c|}{\textbf{0.209}} & \multicolumn{1}{c|}{(0.7793,0.6807)} & \multicolumn{1}{c|}{(0.8113,0.7551)}          & \textbf{(0.8926,0.7991)} \\
\multicolumn{1}{c|}{}                     & \multicolumn{1}{c|}{(300,400,400)} & \multicolumn{1}{c|}{1.408} & \multicolumn{1}{c|}{1.100}          & \multicolumn{1}{c|}{\textbf{0.282}} & \multicolumn{1}{c|}{1.966} & \multicolumn{1}{c|}{0.861}          & \multicolumn{1}{c|}{\textbf{0.208}} & \multicolumn{1}{c|}{(0.7105,0.2885)} & \multicolumn{1}{c|}{(0.6924,0.3386)}          & \textbf{(0.8101,0.7031)} \\ \hline
\end{tabular}}
\end{table}

\section{Conclusions}\label{sec-con}

In this paper, we presented a matrix variant of trace Lasso regularization, and proposed an adaptive sparse CCA model by incorporating the trace Lasso regularization into CCA problem. The proposed model can well cope with the situation in which the data set are correlate.  The adaptive sparse CCA  is further reformulated to an optimization problem on Riemannian manifolds, and a manifold inexact augmented Lagrangian method is then proposed for solving the resulting  optimization problem. Note that the proposed manifold inexact augmented Lagrangian method can be used to solve the general manifold constrained optimization problem: $\min_{U\in\mathcal{M}_1,V\in\mathcal{M}_2}\{f(U,V) + g(U) + h(V)\}$, where $g$ and $h$ is may nonsmooth. Then, An Riemannian Barzilai-Borwein gradient method is adopted for the iteration subproblem of the proposed method, and the global convergence is established under some mild assumptions.   We show that adaptive sparse CCA can significantly improve the performance compared with the  $\ell_1$-regularization based sparse CCA technique (CoLaR) in different simulation settings.

\vskip 10pt

\bibliographystyle{abbrv}
\bibliography{SCCA_Ref}

\begin{thebibliography}{10}

\bibitem{AbsMahSep2008}
P.-A. Absil, R.~Mahony, and R.~Sepulchre.
\newblock {\em Optimization algorithms on matrix manifolds}.
\newblock Princeton University Press, 2009.

\bibitem{beck2017first}
A.~Beck.
\newblock {\em First-order methods in optimization}, volume~25.
\newblock SIAM, 2017.

\bibitem{Boumal2016Global}
N.~Boumal, P.-A. Absil, and C.~Cartis.
\newblock Global rates of convergence for nonconvex optimization on manifolds.
\newblock {\em IMA Journal of Numerical Analysis}, 39(1):1--33, 2018.

\bibitem{Chen2012Structured}
X.~Chen, H.~Liu, and J.~G. Carbonell.
\newblock Structured sparse canonical correlation analysis.
\newblock {\em Applied Intelligence}, 40(2):291--304, 2012.

\bibitem{Correa2008Canonical}
N.~M. Correa, Y.-O. Li, T.~Adali, and V.~D. Calhoun.
\newblock Canonical correlation analysis for feature-based fusion of biomedical
  imaging modalities and its application to detection of associative networks
  in schizophrenia.
\newblock {\em IEEE journal of selected topics in signal processing},
  2(6):998--1007, 2008.

\bibitem{dkk2019ALM}
K.~Deng and P.~Zheng.
\newblock An inexact augmented lagrangian method for nonsmooth optimization on
  riemannian manifold.
\newblock 2019.

\bibitem{Fu2008Image}
Y.~Fu and T.~S. Huang.
\newblock Image classification using correlation tensor analysis.
\newblock {\em IEEE Transactions on Image Processing}, 17(2):226--234, 2008.

\bibitem{gao2017sparse}
C.~Gao, Z.~Ma, H.~H. Zhou, et~al.
\newblock Sparse cca: Adaptive estimation and computational barriers.
\newblock {\em The Annals of Statistics}, 45(5):2074--2101, 2017.

\bibitem{Grave2011Trace}
E.~Grave, G.~Obozinski, and F.~Bach.
\newblock Trace lasso: a trace norm regularization for correlated designs.
\newblock {\em Advances in Neural Information Processing Systems}, pages
  2187--2195, 2011.

\bibitem{Hoerl2000Ridge}
A.~E. Hoerl and R.~W. Kennard.
\newblock Ridge regression: biased estimation for nonorthogonal problems.
\newblock {\em Technometrics}, 42(1):80--86, 2000.

\bibitem{Hotelling1992Relations}
H.~Hotelling.
\newblock {\em Relations Between Two Sets of Variates}.
\newblock Springer New York, 1992.

\bibitem{IannazzoThe}
B.~Iannazzo and M.~Porcelli.
\newblock The riemannian barzilai–borwein method with nonmonotone line search
  and the matrix geometric mean computation.
\newblock {\em IMA Journal of Numerical Analysis}, 38(1):495--517.

\bibitem{Lin2014Correspondence}
D.~Lin, V.~D. Calhoun, and Y.-P. Wang.
\newblock Correspondence between fmri and snp data by group sparse canonical
  correlation analysis.
\newblock {\em Medical image analysis}, 18(6):891--902, 2014.

\bibitem{Lin2013Group}
D.~Lin, J.~Zhang, J.~Li, V.~D. Calhoun, H.-W. Deng, and Y.-P. Wang.
\newblock Group sparse canonical correlation analysis for genomic data
  integration.
\newblock {\em BMC bioinformatics}, 14(1):245, 2013.

\bibitem{Loog2004Dimensionality}
M.~Loog, B.~Van~Ginneken, and R.~P. Duin.
\newblock Dimensionality reduction by canonical contextual correlation
  projections.
\newblock In {\em European Conference on Computer Vision}, pages 562--573.
  Springer, 2004.

\bibitem{lu2013correlation}
C.~Lu, J.~Feng, Z.~Lin, and S.~Yan.
\newblock Correlation adaptive subspace segmentation by trace lasso.
\newblock In {\em Proceedings of the IEEE International Conference on Computer
  Vision}, pages 1345--1352, 2013.

\bibitem{Parkhomenko2008Sparse}
E.~Parkhomenko.
\newblock Sparse canonical correlation analysis.
\newblock {\em Statistical Applications in Genetics \& Molecular Biology},
  8(1):Article 1, 2008.

\bibitem{Parkhomenko2009Sparse}
E.~Parkhomenko, D.~Tritchler, and J.~Beyene.
\newblock Sparse canonical correlation analysis with application to genomic
  data integration.
\newblock {\em Statistical applications in genetics and molecular biology},
  8(1):1--34, 2009.

\bibitem{suo2017sparse}
X.~Suo, V.~Minden, B.~Nelson, R.~Tibshirani, and M.~Saunders.
\newblock Sparse canonical correlation analysis.
\newblock {\em arXiv preprint arXiv:1705.10865}, 2017.

\bibitem{Vinod2006Canonical}
H.~D. Vinod.
\newblock Canonical ridge and econometrics of joint production.
\newblock {\em Journal of econometrics}, 4(2):147--166, 1976.

\bibitem{Waaijenborg2007Quantifying}
S.~Waaijenborg, P.~C.~V. de~Witt~Hamer, and A.~H. Zwinderman.
\newblock Quantifying the association between gene expressions and dna-markers
  by penalized canonical correlation analysis.
\newblock {\em Statistical applications in genetics and molecular biology},
  7(1), 2008.

\bibitem{wang2014robust}
J.~Wang, C.~Lu, M.~Wang, P.~Li, S.~Yan, and X.~Hu.
\newblock Robust face recognition via adaptive sparse representation.
\newblock {\em IEEE transactions on cybernetics}, 44(12):2368--2378, 2014.

\bibitem{wang2015robust}
Y.~Wang, X.~Lin, L.~Wu, W.~Zhang, Q.~Zhang, and X.~Huang.
\newblock Robust subspace clustering for multi-view data by exploiting
  correlation consensus.
\newblock {\em IEEE Transactions on Image Processing}, 24(11):3939--3949, 2015.

\bibitem{Wilms2015Sparse}
I.~Wilms and C.~Croux.
\newblock Sparse canonical correlation analysis from a predictive point of
  view.
\newblock {\em Biometrical Journal}, 57(5):834--851, 2015.

\bibitem{Witten2009A}
D.~M. Witten, R.~Tibshirani, and T.~Hastie.
\newblock A penalized matrix decomposition, with applications to sparse
  principal components and canonical correlation analysis.
\newblock {\em Biostatistics}, 10(3):515--534, 2009.

\bibitem{Witten2009Extensions}
D.~M. Witten and R.~J. Tibshirani.
\newblock Extensions of sparse canonical correlation analysis with applications
  to genomic data.
\newblock {\em Statistical applications in genetics and molecular biology},
  8(1):1--27, 2009.

\bibitem{Yang2014}
W.~Yang, L.~Zhang, and R.~Song.
\newblock Optimality conditions for the nonlinear programming problems on
  riemannian manifolds.
\newblock {\em Pacific Journal of Optimization}, 10(2):415--434, 2014.

\end{thebibliography}
\newpage
\begin{appendix}

\section{Riemannian submanifold}
In this section, we introduce some necessary concepts and definitions of Riemannian optimization, we refer the reader to \cite{AbsMahSep2008} for more details. A $d$-dimensional smooth manifold $\mathcal{M}$ is a Hausdorff and second-countable topological space, where each point has a neighborhood being  locally homeomorphic to the $d$-dimensional Euclidean space  via a
family of charts, and the transition maps of intersecting charts are differentiable. At each point $x\in\mathcal{M}$, a tangent vector $\xi_x$ is defined as a mapping such that there exists a curve $\gamma$ on $\mathcal{M}$ with $\gamma(0) = x$, and satisfying
\[\nn
\xi_x f:=\dot{\gamma}(0)f = \left.\frac{d(f(\gamma(t)))}{dt}\right|_{t=0}
\]
for all $f\in\Im_x\mathcal{M}$. The tangent space $T_x\mathcal{M}$ to $\mathcal{M}$ at $x$ is defined as the set of all tangent vectors at   $x$.  A  Riemannian manifold $(\mathcal{M}, \left< \right>)$ is a real smooth manifold that quipped with a smoothly varying inner product $\left< \right>_x$ on the tangent space $T_x\mathcal{M}$ of each point $x\in\mathcal{M}$. When $\mathcal{M}$ is a Riemannian submanifold of   Euclidean space $\mathcal{E}$, the inner product is defined as Euclidean inner product:$\left<\eta_x,\xi_x\right>_x = \left<\eta_x,\xi_x\right>$, and the corresponding norm induced by the Riemannian metric is given by $\|\xi_x\|_x = \|\xi_x\|$.  A retraction on  manifold $\mathcal{M}$ is a smooth mapping $R: T\mathcal{M}\rightarrow \mathcal{M}$ with the following properties. Let $R_x: T_x\mathcal{M} \rightarrow \mathcal{M}$  be the restriction of  $R$ to $T_x\mathcal{M}$:
\begin{itemize}
  \item $R_x(0_x) = x$, where $0_x$ is zero element of $T_x\mathcal{M}$
  \item $dR_x(0_x) = id_{T_x\mathcal{M}}$,  where $id_{T_x\mathcal{M}}$ is the identity mapping on $T_x\mathcal{M}$
\end{itemize}
To compare two tangent vector at different points, we need  vector transport. The vector transport $\mathcal{T}$ is a smooth mapping with
  $
    T\mathcal{M}\oplus T\mathcal{M} \rightarrow T\mathcal{M}:(\eta_x,\xi_x)\mapsto \mathcal{T}_{\eta_x}(\xi_x)\in T\mathcal{M}
  $
  for any $x\in\mathcal{M}$, where $\mathcal{T}$ satisfies
  \begin{itemize}
    \item for any $\xi_x\in T_x\mathcal{M}$, $\mathcal{T}_{0_x}\xi_x = \xi_x$,
    \item $\mathcal{T}_{\eta_x}(a\xi_x + b\zeta_x)  = a \mathcal{T}_{\eta_x}(\xi_x) + b\mathcal{T}_{\eta_x}(\zeta_x)$.
  \end{itemize}
  Vector transport preserves inner products, i.e. $\left<\mathcal{T}_{\eta_x}(\xi_x),\mathcal{T}_{\eta_x}(\zeta_x)\right>_x = \left<\xi_x,\zeta_x\right>_x$.

\section{Proximal operator and retraction-smooth}
For a proper, convex and low semicontinuous function $g:\mathcal{E}\rightarrow\mathbb{R}$, the proximal operator with  parameter $\mu\geq 0$, denoted by $\mbox{prox}_{\mu g}$,  is defined by
\begin{equation}\label{proximal}
  \mbox{prox}_{\mu g} (v)  :=\arg\min_x \{g(x) + \frac{1}{2\mu} \|x-v\|^2\}.
\end{equation}
 The associated Moreau envelope is defined as a function $M_{\mu g}: \mathcal{E}\rightarrow\mathbb{R}$, given by
\begin{equation}\label{More}
\begin{split}
  M_{\mu g} (v) :& = \min_x \{g(x) + \frac{1}{2\mu} \|x-v\|^2\} \\
                 & = g(\mbox{prox}_{\mu g} (v)) + \frac{1}{2\mu} \|\mbox{prox}_{\mu g} (v)-v\|^2.
  \end{split}
\end{equation}
Note that the Moreau envelope is a continuously differentiable function, even when $g$ is not. The following lemma states this fact.
\begin{lemma}[Theorem 6.60 in \cite{beck2017first}]\label{lem-morea}
  Let $g:\mathcal{E}\rightarrow\mathbb{R}$ be a proper closed and convex function, $\mu\ge 0$. Then $M_{\mu g}$ is $\frac{1}{\mu}$-smooth over $\mathcal{E}$, and for any $v\in\mathcal{E}$ we have
   \begin{equation}\label{Morediff}
  \nabla M_{\mu g}(v) = \frac{1}{\mu}(v - \mbox{prox}_{\mu g} (v)).
\end{equation}
\end{lemma}
Lemma \eqref{lem-morea} states that the Moreau envelope is a continuously differentiable function over Euclidean space $\mathcal{E}$. The next results show the relationship between Retraction smoothness in a submanifold of Euclidean space and smoothness in Euclidean space.

\begin{definition}[Retr-Smooth]
  A function $f:\mathcal{M}\rightarrow\mathbb{R}$ is said to be retraction $L$-smooth if for any $x,y\in\mathcal{M}$,
  \begin{equation}
    f(y) \leq f(x) + \left<\mbox{grad}f(x), \xi\right>_x + \frac{L}{2}\|\xi\|_x^2
  \end{equation}
  where $\xi\in T_x\mathcal{M}$ and $y =R_x(\xi)$.
\end{definition}

Let $\mathcal{M}$ be a Riemannian submanifold of $\mathcal{E}$. The following lemma states that if $f:\mathbb{R}^n \rightarrow \mathbb{R}$ have Lipschitz continuous gradient, then $f$ is also retraction smooth.

\begin{lemma}[Lemma 4 in \cite{Boumal2016Global}]\label{retract-euclidean}
  Let $\mathcal{E}$ be a Euclidean space (for example, $\mathcal{E} = \mathbb{R}^n$) and let $\mathcal{M}$ be a compact Riemannian submanifold of $\mathcal{E}$. Let $Retr$ be a retraction on $\mathcal{M}$. if $f:\mathcal{E}\rightarrow \mathbb{R}$ has Lipschitz continuous gradient in the convex hull of $\mathcal{M}$, then there exists a positive constant $L_g$  such that  the pullbacks $f\circ \mbox{Retr}_x$ satisfies
  \begin{equation}
    f(\mbox{Retr}_{x}(\eta)) \leq f(x) + \left<\eta,\mbox{grad}f(x)\right> + \frac{L_g}{2}\|\eta\|^2
  \end{equation}
  for all $\eta\in T_{x}\mathcal{M}$.
\end{lemma}

\end{appendix}

\end{document}